\documentclass{siamltex}
\usepackage[dvips]{graphicx}
\usepackage{amsmath}
\usepackage{latexsym}
\usepackage{amssymb}
\usepackage{color}
\usepackage{algorithm,algorithmic}
 \usepackage[usenames,dvipsnames]{pstricks}
 \usepackage{pst-grad}
 \usepackage{pst-plot} 
 \usepackage{mathptmx}      
 \usepackage{latexsym}
 \usepackage{fancyhdr}
 
 \usepackage{color}
 \usepackage{hyperref}
 
 \makeatletter
 \newcommand{\ssymbol}[1]{^{\@fnsymbol{#1}}}
 \makeatother

 		\title{Tensor  GMRES and  Golub-Kahan Bidiagonalization methods via  the  Einstein product with applications to image and  video processing}
 			\author{ M. El Guide\thanks{Centre for Behavioral Economics and Decision Making(CBED), FGSES, Mohammed VI Polytechnic University, Green City, Morocco.}  \and  A. El Ichi \footnotemark[4] \thanks{Department of Mathematics University Mohammed V Rabat, Morocco} \and   K. Jbilou\footnotemark[1] \thanks{LMPA, 50 rue F. Buisson, ULCO Calais, France; jbilou@univ-littoral.fr } \and F.P.A. Beik\thanks{Department of Mathematics, Vali-e-Asr University of Rafsanjan,\ P.O. Box 518, Rafsanjan, IranA.} }
 		
\begin{document}
 		 	
 		 	\maketitle


 		\begin{abstract}    
 			In the present paper, we are interested in developing  iterative Krylov subspace methods in tensor structure to solve a class of multilinear systems via Einstein product. In particular, we develop global variants of the GMRES and Gloub--Kahan bidiagonalization processes in tensor framework. We further consider the case that mentioned equation may be possibly corresponds to a discrete ill-posed problem. Applications arising from color image and video restoration are included. 		
 			\end{abstract}

 	{\bf keywords}:
 			{Arnoldi process, Golub--Kahan, ill-posed problem bidiagonalization, tensor equation, Einstein product, Video processing.}

 	\medskip 
 	
 	
 	\section{Introduction}
 	In this paper, we are interested in approximating the solution of  following tensor equation
 	\begin{equation}\label{eq1}
 \mathcal A \ast_N \mathcal X = \mathcal C,
 	\end{equation}
 	where  $ \mathcal A\in \mathbb{R}^{I_1 \times  \ldots \times I_N \times I_1 \times \ldots \times I_N}$ and  $ \mathcal C\in \mathbb{R}^{I_1 \times  \ldots \times I_N \times J_1 \times \ldots \times J_M}$ are known and $ \mathcal X\in \mathbb{R}^{I_1 \times  \ldots \times I_N \times J_1 \times \ldots \times J_M}$is an unknown tensor to be determined.
 We can also consider the  least-squares problem
 	\begin{equation*}
 \displaystyle \min \Vert \mathcal A \ast_N \mathcal X-\mathcal C \Vert_F.
 \end{equation*}	

\noindent	Tensor equations  arise
in many application of modern sciences, e.g., engineering \cite{qllz}, signal processing \cite{lb},  data mining \cite{lxnm}, tensor complementarity problems\cite{lzqlxn}, computer vision\cite{vt1, vt2} and  as a result have been extensively studied in the literature \cite{comon,smile, kroon}. The most recent tensor approaches used  for numerically
solving PDEs have been investigated in \cite{dwwm}. For those applications, we have to take advantage of this multidimensional structure to build rapid and robust methods for solving the related problems.
For an extensive literature on tensors one can see  for example the good papers in  \cite{kimler1, kolda1}. Over the years many specialized methods for solving tensor problems of type (\ref{eq1}) have been developed, see e.g. \cite{huang1} for tensor forms of the Arnoldi and Lanczos processes for well-posed problems. Huang et al. \cite{huang1} pointed out that tensor equations of the form \eqref{eq1} appear in continuum physics, engineering, isotropic and anisotropic elastic models. Multilinear systems of the form \eqref{eq1} may also arise from discretization of the high-dimensional Poisson problem using finite difference approximations \cite{brazell,huang1}. 

%

In the current paper, we are interested in  developing robust and fast iterative Krylov subspace methods via Einstein product to solve regularized problems originating from color image and  video processing applications. Standard and global Krylov subspace methods are suitable  when dealing with grayscale images, e.g, \cite{belguide, belguide2, reichel1, reichel2},  while  Krylov subspace methods can handle similar applications when the blurring linear operator can be decomposed in Kroncker product of two matrices; see \cite{belguide, belguide2}. However, much work has to be done to numerically solve problems related to multi channel images (e.g. color images, hyper-spectral images and videos).  We show that modelling these problems in the form of tensor equation (\ref{eq1}) make it possible to develop iterative Krylov subspace methods more appealing and allows to significantly reduce the overall computational complexity.

The remainder of paper is organized as follows: We shall first in Section \ref{sec:section2} by presenting some symbols and notations used throughout paper.  Section \ref{sec3} includes reviewing the adaptation of Tikhonov regularization for tensor equation (\ref{eq1}). Then we propose GMRES and Global Golub--Kahan methods via Einstein in conjunction with Tikhonov regularization.  On the basis of Point Spread Function (PSF),  in Section \ref{sec4}, we propose a tensor formulation in the form of (\ref{eq1}) that describes the blurring of color image and video processing. Numerical examples  are reported on restoring blurred and noisy color images and videos. Concluding remarks can be found in Section \ref{sec5}.

\section{Definitions and Notations}\label{sec:section2}   In this section, we briefly review some concepts and notions
that are used throughout the paper. 	A \textit{tensor} is a multidimensional array of data and a natural extension of scalars, vectors and matrices to a higher order, a scalar is a $0^{th}$\textbf{ order} tensor, a vector is a  $1^{th}$\textbf{ order} tensor and a matrix is   $2^{th}$\textbf{ order} tensor. The tensor \textbf{ order}  is  the number of its indices, which is called \textit{modes} or \textit{ways}. For a given N-mode tensor $  \mathcal {X}\in \mathbb{R}^{I_{1}\times I_{2}\times I_{3}\ldots \times I_{N}}$,
 the notation $x_{i_{1},\ldots,i_{N}}$ (with $1\leq i_{j}\leq I_{j},\; j=1,\ldots N $) stands for element $\left(i_{1},\ldots,i_{N} \right) $ of the tensor $\mathcal {X}$. Corresponding to a given tensor $ \mathcal {X}\in \mathbb{R}^{I_{1}\times I_{2}\times I_{3}\ldots \times I_{N}}$, the notation 
\[\mathop{\mathcal{X}_{\underbrace{::\cdots:}k}}\limits_{\tiny{(N-1)-times}},\quad k=1,2,\ldots,I_N,\]
denotes a tensor in $\mathbb{R}^{I_{1}\times I_{2}\times I_{3}\ldots \times I_{N-1}}$ which is obtained by fixing the last index and is called frontal 
slice; see \cite{kimler1,kolda1} for more details.
 Throughout this work, vectors and matrices are respectively denoted by
lowercase and capital letters, and tensors of  higher order  are represented by  calligraphic   letters.
\medskip

We first recall the definition of $n$-mode tensor product with a matrix; see \cite{kolda1} .

\begin{definition}
The $n$-mode product of the tensor 	$\mathcal {A}=[a_{i_1i_2 \ldots i_n} ]  \in \mathbb{R}^{I_{1} \times I_2\times \ldots \times I_N}$ and the matrix $U=[u_{ji_n}] \in \mathbb{R}^{J \times I_n}$ is denoted by $\mathcal {A} \times_n U$ is a tensor of order $I_1\times I_2 \times \ldots \times I_{n-1} \times J \times I_{n+1} \times \ldots \times I_N$ and its entries are defined by
\begin{equation*}
(\mathcal {A} \times_n U)_{i_1i_2\ldots i_{n-1}ji_{n+1}\ldots i_N}= \displaystyle \sum_{i_n=1}^{I_N} a_{i_1i_2\ldots i_N} u_{ji_n}
\end{equation*}
\end{definition}

\medskip 

\noindent The $n$-mode product of the tensor $\mathcal {A} \in \mathbb{R}^{I_{1} \times I_2\times \ldots \times I_N}$ with the vector $v=[v_{i_n}] \in \mathbb{R}^{I_n}$  is an $(N-1)$-mode tensor denoted by $\mathcal {A} \bar {\times} v$ whose elements are given by 
$$ (\mathcal {A} \bar {\times} v )_{i_1\ldots i_{n-1}i_{n+1}\ldots i_N}= \displaystyle \sum_{i_n}{}x_{i_1i_2\ldots i_N} v_{i_n}.$$ 

\medskip 


 \noindent 	Next, we recall the definition and some properties  of the tensor Einstein product which is an  extension of the matrix product; for more details see \cite{brazell}
 	
 	\begin{definition}\cite{einstein}
 	
 	Let $\mathcal {A} \in \mathbb{R}^{I_{1}\times I_{2}\times \ldots \times I_{L}\times   K_{1}\times K_{2}\times \ldots \times K_{N}} $, $\mathcal {B} \in \mathbb{R}^{K_{1}\times K_{2}\times \ldots \times K_{N}\times J_{1}\times J_{2}\times \ldots \times J_{M}} $, the Einstein product of tensors $\mathcal {A}$ and $\mathcal {B}$ is a tensor of size   $  \mathbb{R}^{I_{1}\times I_{2}\times \ldots \times I_{L}\times J_{1}\times J_{2}\times \ldots \times J_{M}} $ whose elements are defined by
 	$$ (\mathcal {A}\ast_{N}\mathcal {B})_{i_{1}\ldots i_{L} j_{1} \ldots j_{M}}=       \sum_{k_{1},\ldots,k_{N}}a_{i_{1}\ldots i_{L} k_{1} \ldots k_{N} }b_{k_{1} \ldots k_{N}j_{1} \ldots j_{M}}. $$
 	
 	\end{definition} 
 	
 	\noindent Given a tensor $\mathcal {A} \in \mathbb{R}^{I_{1}\times I_{2}\times \ldots \times I_{N}\times   J_{1}\times J_{2}\times \ldots \times J_{M}} $, the tensor  
 	$\mathcal {B} \in \mathbb{R}^{J_{1}\times J_{2}\times \ldots \times J_{M}\times   I_{1}\times I_{2}\times \ldots \times I_{N}} $ the transpose of  $\mathcal {A}$,
 	if $b_{i_1\ldots i_Mj_1 \ldots j_m}= a_{j_1 \ldots j_Ni_1 \ldots i_M}$. We denote the transpose of  $\mathcal {A}$ by $\mathcal {A}^T$.

   \noindent  A tensor $\mathcal{D}=[d_{i_{1},\ldots,i_{M},j_{1},\ldots,j_{N}}]\in \mathbb{R}^{I_1 \times \cdots \times I_N\times J_1 \times \cdots \times J_N}$ is said to be diagonal  if all of its entries are equal to zero except for $d_{i_1\ldots i_Ni_1\ldots i_N}$. In the case $d_{i_1\ldots i_Ni_1\ldots i_N}=1$, the tensor $\mathcal{D}$ is called diagonal and denoted by $\mathcal{I}_N$. We  further use the notation $\mathcal{O}$ for a the tensor having all its entries equal to zero.

   \begin{definition}
 Let  $\mathcal{A} \in \mathbb{R}^{I_{1}\times I_{2}\times \ldots \times I_{N}\times   I_{1}\times I_{2}\times \ldots \times I_{N}} $. The tensor  $\mathcal{A}$ is invertible if there exists a tensor $\mathcal {X} \in  \mathbb{R}^{I_{1}\times I_{2}\times \ldots \times I_{N}\times   I_{1}\times I_{2}\times \ldots \times I_{N}} $ such that
 	$\mathcal{A}\ast_{N}\mathcal {X}=\mathcal{X}\ast_{N}\mathcal {A}=\mathcal {I}_N.$
 \end{definition}
	
 \medskip
 \noindent The trace of an even-order  tensor $\mathcal{A}\in \mathbb{R}^{I_{1}\times I_{2}\times I_{3}\ldots \times I_{N}\times   I_{1}\times I_{2}\times I_{3}\ldots \times I_{N}}$ is given by 
$$
  tr(\mathcal{A})=\sum_{i_{1} \ldots i_{N}} a_{i_{1} \ldots i_{N} i_{1} \ldots i_{N}}.
$$

\begin{definition}
The inner product of two same size tensors $\mathcal{X}, \mathcal{Y} \in \mathbb{R}^{I_1\times I_2\times \cdots \times I_N}$ is defined by
\begin{equation*}
\left\langle {\mathcal{X}, \mathcal{Y}} \right\rangle  = \sum\limits_{i_1  = 1}^{I_1 } {\sum\limits_{i_2  = 1}^{I_2 } {\ldots \sum\limits_{i_N  = 1}^{I_N } {x_{i_1 i_2 \cdots i_N }^{} } } } y_{i_1 i_2 \cdots i_N }^{}.
 	\end{equation*}
Notice that for even order tensors $\mathcal {X},\mathcal {Y}\in \mathbb{R}^{I_{1}\times I_{2}\times I_{3}\ldots \times I_{N}\times   J_{1}\times J_{2}\times J_{3}\ldots \times J_{M}}$, we have
 	\begin{equation*}
 	\langle \mathcal {X},\mathcal {Y} \rangle=tr(\mathcal {X}^{T}\ast_{N}\mathcal {Y}) 
 	\end{equation*}
 	where $\mathcal {Y}^{T}\in \mathbb{R}^{J_{1}\times J_{2}\times J_{3}\ldots \times J_{M}\times I_{1}\times I_{2}\times I_{3}\ldots \times I_{N}}$ denote de transpose of $\mathcal {Y}.$ \\
 The Frobenius norm of the tensor  $\mathcal {X}$ is given by 	
 \begin{equation}\label{inorm11}
   ||\mathcal {X}||_F=\left\langle {\mathcal{X}, \mathcal{Y}} \right\rangle= \displaystyle \sqrt{tr(\mathcal {X}^{T}\ast_{N}\mathcal {X}) }.
    	\end{equation}
 		\end{definition}
 	
%
%

\noindent The two tensors   $\mathcal {X},\mathcal {Y}\in \mathbb{R}^{I_{1}\times I_{2}\times \ldots \times I_{N}\times   J_{1}\times J_{2}\times \ldots \times J_{M}}$ are orthogonal iff 
 $	\langle \mathcal {X},\mathcal {Y} \rangle=0$.\\
 
%
%

\medskip
\noindent In \cite{beik1}, the $\boxtimes^N$ product between $N$-mode tensors $\mathcal{X}\in \mathbb{R}^{I_1\times I_2 \times \cdots \times I_{N-1} \times  I_N}$ and $\mathcal{Y}\in \mathbb{R}^{I_1\times {I}_2 \times \cdots \times I_{N-1} \times \tilde{I}_N}$ is defined as an $I_N \times \tilde{I}_N$ matrix whose $(i,j)$-th entry is
\[
[\mathcal{X} \boxtimes^N \mathcal{Y}]_{ij}=\text{tr} (\mathcal{X}_{{::\dots:}i} \boxtimes^{N-1} \mathcal{Y}_{{::\dots:}j}),\qquad N=3,4,\ldots,
\]
where
\[\mathcal{X} \boxtimes^2 \mathcal{Y}= \mathcal{X}^T \mathcal{Y}, \qquad \mathcal{X}\in \mathbb{R}^{I_1\times I_2}, \mathcal{Y}\in \mathbb{R}^{I_1\times \tilde{I}_2}.\]
\noindent  Basically, the product $\mathcal{X} \boxtimes^N \mathcal{Y}$ is the contracted product of $N$-mode tensors $\mathcal{X}$ and $\mathcal{Y}$ along the first $N-1$ modes.\\
It is immediate to see that for $\mathcal{X}, \mathcal{Y} \in \mathbb{R}^{I_1\times I_2 \times \cdots \times I_N}$, we have
\begin{equation*}
\left\langle {\mathcal{X}, \mathcal{Y}} \right\rangle =\text{tr}(\mathcal{X} \boxtimes^N \mathcal{Y}),\qquad N=2,3,\ldots,
\end{equation*}
and
\[\left\| \mathcal{X} \right\|^2= \text{tr} (\mathcal{X} \boxtimes^N \mathcal{X})=\mathcal{X} \boxtimes^{(N+1)} \mathcal{X},\]
for $\mathcal{X}\in \mathbb{R}^{I_1\times I_2 \times \cdots \times I_N}$.\\ We end  the current subsection by recalling the following useful proposition from \cite{beik1}.
\begin{proposition} \label{p2} Suppose that $\mathcal{B}\in \mathbb{R}^{I_1\times I_2 \times \cdots \times I_N\times m}$ is an $(N+1)$-mode tensor with the column tensors $\mathcal{B}_1,\mathcal{B}_2,\ldots,\mathcal{B}_m\in \mathbb{R}^{I_1\times I_2 \times \cdots \times I_N}$ and
	$z=(z_1,z_2,\ldots,z_m)^T\in \mathbb{R}^m$. For an arbitrary $(N+1)$-mode tensor $\mathcal{A}$ with $N$-mode column tensors $\mathcal{A}_1,\mathcal{A}_2,\ldots,\mathcal{A}_m$, the following statement holds
	\begin{equation}
	\mathcal{A} \boxtimes^{(N+1)} (\mathcal{B} \bar{\times}_{_{N+1}} z) = (\mathcal{A} \boxtimes^{(N+1)} \mathcal{B}) z.
	\end{equation}
\end{proposition}

\section{Krylov subspace methods via Einstein product}\label{sec3}

In this section,  we recall the tensor global Arnoldi and propose iterative methods based on Global Arnoldi and Global Golub--Kahan bidiagonlization (GGKB) combined with Tikhonov regularization that are applicable to the restoration of a color images and videos from an available blur- and noise-contaminated versions.

 \subsection{Tikhonov regularization}
 Many applications require the solution of several ill-conditioning systems of equations of the form (\ref{eq1})
 with a right hand side contaminated by an additive error,
 \begin{equation}\label{eq1err}
 \mathcal{A}\ast_N\mathcal X = \mathcal C+\mathcal E,
 \end{equation}
 where $\mathcal E$ is the matrix of error terms that may stem from measurement and discretization errors. An ill-posed tensor equation may appear in color image restoration, video restoration, and when solving some partial differential equations in several space dimensions. In order to diminish the effect of the noise in the data, we replace the original problem  by a stabilized one. One of the most popular regularization
 methods is due to Tikhonov \cite{tikhonov}. 
 Tikhonov regularization problem to solve \eqref{eq1err} is given by
 \begin{equation}\label{Tikhrd}
 \mathcal{X}_\mu=\text{arg}\underset{ \mathcal {X}}{ \text{min}} \left(\| \mathcal {A} \ast_N  \mathcal {X}  -\mathcal C\|_{F}^2+\mu\|\mathcal{X}\|_F^2 \right),
 \end{equation}
 The choice of $\mu$ affects how sensitive $\mathcal{X}_\mu$ is to the error $\mathcal{E}$ in the contaminated right-hand side. Many techniques for choosing a suitable value of $\mu$ have been analyzed and illustrated in the literature; see, e.g., \cite{wahbagolub} and references therein. In this paper we use the discrepancy principle and the Generalized Cross Validation (GCV) techniques.

\subsection{Global GMRES method via Einstein product}
 Let $\mathcal {A} \in \mathbb{R}^{I_{1}\times \ldots \times I_{N}\times   I_{1}\times \ldots \times I_{N}}$ be a square tensor and   $\mathcal {V} \in \mathbb{R}^{I_{1}\times I_{2}\times \ldots \times I_{N}\times   J_{1}\times K_{2}\times \ldots \times J_{M}}$. The $m$-th tensor Krylov subspace is defined by
\begin{equation}\label{krylov1}
\mathcal {K}_m(\mathcal {A},\mathcal {V})=span\{\mathcal {V},\mathcal {A},\mathcal {V},\ldots,\mathcal {A}^{m-1}(\mathcal {V})) \},
\end{equation}
where $\mathcal {A}^i(\mathcal {V})=\mathcal {A} (\mathcal {A}^{i-1}(\mathcal {V}))$. 
The global Arnoldi process for matrix case was proposed in \cite{jbilou1}.
The algorithm for constructing orthonormal basis of \eqref{krylov1} can be given  as follows: (see \cite{beik1,huang1,jbilou1})
 
 \begin{algorithm}
 	\caption{Global Arnoldi process via Einstein product}\label{alg1}
 	\begin{enumerate}
 	
 	\item Inputs: A tensor $\mathcal {A} \in \mathbb{R}^{I_{1}\times I_{2}\times \ldots \times I_{N}\times   I_{1}\times K_{2}\times \ldots \times K_{N}}$, and a tensor   $\mathcal {V} \in \mathbb{R}^{I_{1}\times I_{2}\times \ldots \times I_{N}\times   J_{1}\times K_{2}\times \ldots \times J_{M}}$ and the integer $m$.
 	\item Set $\beta=\Vert \mathcal {V} \Vert_F$ and $\mathcal {V}_1=\mathcal {V} /\beta$.
 	\item For $j=1,\ldots,m$
 	\item $\mathcal {W}=\mathcal {A} \ast_N\mathcal {V}_j$
 	\item for $ i=1,\ldots,j$.
 	\begin{itemize}
 	\item $h_{ij}=\langle \mathcal {V}_i,\mathcal {W} \rangle$,
 	\item $\mathcal {W}= \mathcal {W}-h_{ij}\mathcal {V}_i$
 	\end{itemize}
 \item endfor
 \item $h_{j+1,j}=\Vert \mathcal {W} \Vert_F$. If $h_{j+1,j}=0$, stop; else
 \item $\mathcal {V}_{j+1}=\mathcal {W}/h_{j+1j}$.
 \item EndFor
 	
 \end{enumerate}
\end{algorithm}

\medskip 
\noindent Let $\widetilde {H}_m$ be the upper $(m+1 \times m)$ Hessenberg   matrix whose entries are the $h_{ij}$ from Algorithm \ref{alg1} and let $H_m$ be the matrix obtained from $\widetilde {H}_m$ by deleting the last row. Then, it is not difficult to verify that the
 $\mathcal {V}_i$'s obtained from Algorithm \ref{alg1} form an orthonormal basis of the tensor Krylov subspace $\mathcal {K}_m(\mathcal {A},\mathcal {V})$. Analogous to \cite{beik1,jbilou1}, we can prove the following proposition.
 
 \begin{proposition}
 	Let $\mathbb{V}$ be the $(M+N+1)$-mode tensor with frontal slices $\mathcal {V}_1,\mathcal {V}_2,\ldots,\mathcal {V}_m$ and $\mathbb{W}_m$ be the $(M+N+1)$-mode tensor with frontal slices $\mathcal {A} \ast_N \mathcal {V}_1,\ldots,\mathcal {A} \ast_N \mathcal {V}_m$. Then
 	\begin{eqnarray}
 	 \mathbb{W}_m& = &\mathbb{V}_{m+1} \times_{(M+N+1)} {\widetilde H}_m^T\label{lin1}\\
 \nonumber	& = &\mathbb{V}_m \times_{(M+N+1)}  H_m^T + h_{m+1,m}\, \mathcal{L}\times_{(M+N+1)} E_m, 
 	\end{eqnarray}
 where $E_m=[0,0,\ldots,0,e_m]$ with $e_m$ is the $m$-th column of the identity matrix $I_m$ and $\mathcal{L}$ is an $(M+N+1)-$mode whose frontal slices are all zero except that last one being equal to
 	\end{proposition}
 
 \noindent Let $\mathcal {A} \in \mathbb{R}^{I_{1}\times I_{2}\times \ldots \times I_{N}\times   I_{1}\times I_{2}\times \ldots \times I_{N}}$ and $\mathcal {C}  \in \mathbb{R}^{I_{1}\times I_{2}\times \ldots \times I_{N}\times   J_{1}\times J_{2}\times \ldots \times J_{M}}$.  Consider now the linear system of tensor equation 
 \begin{equation}\label{sys1}
 \mathcal {A} \ast_N  \mathcal {X}= \mathcal {C}.
 \end{equation}
 
 \noindent Using Algorithm \ref{alg1}, we can propose the global GMRES method to solve the problem \eqref{sys1}. As for the global GMRES, we seek for an approximate solution $\mathcal {X}_m$, starting from $\mathcal {X}_0$ such that  $\mathcal {X}_m \in \mathcal {X}_0+ \mathcal {K}_m(\mathcal {A},\mathcal {V})$ and by solving the minimization problem
 \begin{equation}
 \label{tgmres1}
 \Vert \mathcal{R}_m \Vert_F= \displaystyle \min_{\mathcal {X} \in \mathcal {X}_0+ \mathcal {K}_m(\mathcal {A},\mathcal {V})}  \Vert  \mathcal {C}-\mathcal {A} \ast_N  \mathcal {X} \Vert_F.
 \end{equation}
 where $\mathcal{R}_m=\mathcal {C}-\mathcal {A} \ast_N  \mathcal {X}$.
 
 \noindent Let $m$ steps of Algorithm \ref{alg1} has been performed. Given an initial guess $\mathcal {X}_0$, we set
 \begin{equation}\label{tgmres2}
\mathcal {X}_m=\mathcal {X}_0+ \mathbb{V}_m \bar{\times}_{(M+N+1)} y_m,
 \end{equation}
which results $\mathcal{R}_m=\mathcal {R}_0-\mathbb{W}_m \bar{\times}_{(M+N+1)} y_m$.
 Using the relations \eqref{lin1}, from Proposition \ref{p2} it immediate to observe that
 \begin{eqnarray*}
 \Vert  \mathcal {C}-\mathcal {A} \ast_N  \mathcal {X}_m \Vert_F & = & \Vert \mathbb{V}_m \boxtimes^{(M+N+1)}( \mathcal {C}-\mathcal {A} \ast_N  \mathcal {X}_m) \Vert_2\\
  & = &\Vert \mathbb{V}_m \boxtimes^{(M+N+1)}(\mathcal {R}_0-\mathbb{W}_m \bar{\times}_{(M+N+1)} y_m) \Vert_2\\
   & = &\Vert \beta e_1^{m+1} -\mathbb{V}_m \boxtimes^{(M+N+1)}(\mathbb{W}_m \bar{\times}_{(M+N+1)} y_m) \Vert_2\\
   & = &\Vert \beta e_1^{m+1} -(\mathbb{V}_m \boxtimes^{(M+N+1)}\mathbb{W}_m) y_m) \Vert_2.
 \end{eqnarray*}
Therefore, $y_m$ is determined as follows:
 \begin{equation}
 \label{tgmres3}
y_m=\arg \min_y \Vert \beta e_1^{m+1} -\widetilde H_m y\Vert_2.
\end{equation}

\noindent The relations \eqref{tgmres2} and  \eqref{tgmres3} define the tensor global GMRES (TG-GMRES). 
\noindent Setting $\mathcal {X}_0=0$ and using the relations \eqref{tgmres1}, \eqref{tgmres2} and \eqref{tgmres3} it follows that instead of solving the  problem \eqref{Tikhrd} we can consider the following  low dimensional Tikhonov regularization problem

\begin{equation}\label{tikho2}
\Vert \beta e_1^{m+1} -\widetilde H_m y\Vert_2^2 + \mu \Vert y \Vert_2^2.
\end{equation}

\medskip 

\noindent The solution of the problem \eqref{tikho2} is given by

\begin{equation}
\label{min}
y_{m,\mu}= \arg \min\left  \Vert \left ( \begin{array}{ll}
\widetilde H_m\\
\sqrt{\mu} I
\end{array}\right ) y - \left ( \begin{array}{ll}
\beta e_1^{m+1}\\
0
\end{array}\right ) 
\right \Vert_2.
\end{equation}
The minimizer $y_{m,\mu}$  of the problem \eqref{min} is computed as the solution of the linear system of equations
\begin{equation}{\label{min1}}
\widetilde H_{m,\mu} y=\widetilde H_m^T\beta e_1^{m+1}
\end{equation}
where $\widetilde H_{m,\mu}= (\widetilde H_m^T \widetilde H_m+ \mu I)$.

\noindent Notice that the Tikhonov problem \eqref{tikho2} is a matrix one with small dimension as $m$ is generally small. Hence it can be solved by some techniques such as the GCV method \cite{golubwahba} or the L-curve criterion \cite{hansen1,hansen2,reichel1,reichel2}. \\

An appropriate selection of the regularization parameter $\mu$
is important in  Tikhonov regularization. Here we can use  the
generalized cross-validation (GCV) method
\cite{bouh,golubwahba,wahbagolub}. For this method, the regularization
parameter is chosen to minimize the GCV function
$$GCV(\mu)=\frac{||\widetilde H_m y_{m,\mu}-\beta e_1^{m+1}||_2^2}{[tr(I-\widetilde H_m  \widetilde H_{m,\mu}^{-1}\widetilde H_m^T)]^2}=\frac{||(I-\widetilde H_m \widetilde H_{m,\mu}^{-1} \widetilde H_m^T)\beta e_1^{m+1}||_2^2}{[tr(I-H_m H_{m,\mu}^{-1} \widetilde H_m^T)]^2}$$ where $\widetilde H_{m,\mu}= (\widetilde H_m^T \widetilde H_m+ \mu I)$ and ${y}_{m, \mu}$ is the solution of
(\ref{min1}). As the projected problem we are dealing with is of small size, we cane use the SVD decomposition of $\widetilde H_m$ to obtain a more simple and computable expression of $GCV(\mu)$. Consider the SVD decomposition of $\widetilde H_m=U\Sigma V^T$. Then the GCV function could be expressed as (see \cite{wahbagolub})

\begin{equation}
\label{gcv2}
GCV(\mu)=\frac{\displaystyle
	\sum_{i=1}^m(\frac{\tilde
		g_i}{\sigma_i^2+\mu})^2}{\displaystyle\Bigl(\sum_{i=1}^m
	\frac{1}{\sigma_i^2+\mu}\Bigr)^2},
\end{equation}
where $\sigma_i$ is the $i$th singular value of the matrix
$\widetilde H_m$ and $\tilde g= \beta_1 U^T  e_1^{m+1}$.

In  the practical implementation, it's more convenient to use a restarted version of the global GMRES. As the number of outer iterations increases, it is possible to compute the $m$-th residual without forming the solution. This is described in the following theorem.
\begin{proposition}
At step $m$, the residual $\mathcal{R}_{m}=\mathcal{C}-\mathcal{A}\ast_N \mathcal{X}_{m}$ produced by the tensor global GMRES method for solving (\ref{eq1}) has the following expression
\begin{equation}\label{resex}
\mathcal{R}_m=\mathbb{V}_{m+1} \bar{\times}_{(M+N+1)}\left(\gamma_{m+1}Q_me_{m+1}\right),
\end{equation}
where $Q_m$ is the unitary matrix obtained by QR decomposition of the upper Hessenberg matrix $\widetilde{H}_{m}$ and $\gamma_{m+1}$ is the last component of the vector $\beta Q_{m}^{\mathrm{T}} e_{m+1}$ in which $\beta=\|\mathcal{R}_0\|_F$  and $e_{\ell}\in \mathbb{R}^{\ell}$ is the last column of identity matrix. Furthermore,
\begin{equation}\label{resnrm}
    \left\|\mathcal{R}_{m}\right\|_{F}=\left|\gamma_{m+1}\right|
\end{equation}

\end{proposition}
\begin{proof}
    At step $m$, the residual $\mathcal{R}_m=\mathcal {R}_0-\mathbb{W}_m \bar{\times}_{(M+N+1)} y_m$ can be expressed as
    \begin{eqnarray*}
    \mathcal{R}_m & = & \mathcal {R}_0-(\mathbb{V}_{m+1} \times_{(M+N+1)} {\widetilde H}_m^T)\bar{\times}_{(M+N+1)}y_m\\
    &=&\mathcal {R}_0-\mathbb{V}_{m+1} \bar{\times}_{(M+N+1)} ({\widetilde H}_my_m)\
    \end{eqnarray*}
    by considering the QR decomposition $\widetilde{H}_{m}=Q_{m}\widetilde{U}_m$ of the $(m + 1) \times m$ matrix $\widetilde{H}_{m}$, we get
    $$\mathcal{R}_m=\mathcal {R}_0-\mathbb{V}_{m+1} \bar{\times}_{(M+N+1)} (Q_{m}\widetilde{U}_my_m).$$
    Straightforward computations show that
    \begin{eqnarray*}
    \|\mathcal{R}_m\|_F^2 & = & \|\mathcal {R}_0-\mathbb{V}_{m+1} \bar{\times}_{(M+N+1)} (Q_{m}\widetilde{U}_my_m) \|_F^2\\
    & = & \|\mathbb{V}_m \boxtimes^{(M+N+1)}(\mathcal {R}_0-\mathbb{V}_{m+1} \bar{\times}_{(M+N+1)} (Q_{m}\widetilde{U}_my_m)) \|_2^2\\
     & = & \|Q_{m}(Q_m^T\beta e_1^{m+1}-\widetilde{U}_my_m) \|_2^2\\
     & = & \|Q_m^T\beta e_1^{m+1}-\widetilde{U}_my_m \|_2^2\\
     & = & \|z_m-\widetilde{U}_my_m \|_2^2 + \left|\gamma_{m+1}\right|^2
    \end{eqnarray*}
    where $z_m$ denotes vector obtained by deleting the last component of $Q_m^T\beta e_1^{m+1}$. Since $y_m$ solves problem (\ref{tgmres3}), it follows that $y_m$ is the solution of $\widetilde{U}_my_m =z_m$, i.e.,
    \[
    \|z_m-\widetilde{U}_my_m \|_2=0.
    \]
Note that $\mathcal{R}_m$ can be written in the following form
   \begin{eqnarray*}
 \mathcal{R}_m & = &\beta \mathbb{V}_{m+1} \bar{\times}_{(M+N+1)} e_1^{m+1}-\mathbb{V}_{m+1} \bar{\times}_{(M+N+1)} (\widetilde{H}_my_m) \\
 & = & \mathbb{V}_{m+1} \bar{\times}_{(M+N+1)} (\beta e_1^{m+1}- \widetilde{H}_my_m)\\
 & = & \mathbb{V}_{m+1} \bar{\times}_{(M+N+1)} (Q_m(Q_m^T\beta e_1^{m+1}- \widetilde{U}y_m))\\
 & = & \mathbb{V}_{m+1} \bar{\times}_{(M+N+1)} (Q_m\gamma_{m+1}e_{m+1}).
   \end{eqnarray*}
Now the result follows immediately from the above computations. 
\end{proof}

The tensor form of global GMRES  algorithm for solving  (\ref{eq1}) is summarized as
follows:
  \vspace{0.3cm}
 \begin{algorithm}[!h]
	\caption{Global GMRES method via Einstein product for Tikhonov regularization}\label{alg11}
	\begin{enumerate}
\item {\bf Inputs} The tensors $\mathcal {A}$, $\mathcal {C}$, initial guess $\mathcal{X}_0$, a tolerance $\varepsilon$, number of iterations between restarts $m$ and \textbf{Maxit}: maximum number of outer iterations.
\item Compute $\mathcal{R}_0=\mathcal C- \mathcal {A} \ast_N  \mathcal {X}_0$, set $\mathcal{V}=\mathcal{R}_0$ and $k=0$
\item Determine the orthonormal frontal slices $\mathcal{V}_1,\ldots,\mathcal{V}_m$ of $\mathbb{V}_{m}$, and the upper Hessenberg matrix  $\widetilde{H}_m$ by applying Algorithm \ref{alg1} to the pair $\left(\mathcal{A}, \mathcal{ V}\right)$. 
\item  Determine $\mu_{k}$  as the parameter minimizing the GCV function  given by (\ref{gcv2})
\item  Determine $y_m$ as the solution of low-dimensional Tikhonov regularization problem (\ref{tikho2}) and set $\mathcal {X}_m=\mathcal {X}_0+ \mathbb{V}_m \bar{\times}_{(M+N+1)} y_m$ \item  If $\left| \gamma_{m+1}\right|_{F}<\varepsilon$ or $k>\textbf{Maxit}$; Stop \\
else: set $\mathcal X_{0}=\mathcal{X}_{m}$, $k=k+1,$ Goto 2
\end{enumerate}
\end{algorithm}
\subsection{Golub--Kahan method via Einstein}
Instead of finding orthonormal basis for the Krylov subspace and using GMRES method, one can apply oblique projection schemes based on biorthogonal bases for $\mathcal {K}_m(\mathcal {A},\mathcal {V})$ and  $\mathcal {K}_m(\mathcal {A}^T,\mathcal {W})$;  see \cite{jbilou2} for instance.

Here, we exploit the tensor Golub--Kahan algorithm via the Einstein product. It should be commented here that the Golub--Kahan algorithm has been already examined for solving ill-posed Sylvester and Lyapunov tensor equations  with applications to color image restoration \cite{beik2}. \\
Let tensors $\mathcal {A} \in \mathbb{R}^{I_{1}\times \ldots \times I_{N}\times   I_{1}\times \ldots \times I_{N}}$, $\mathcal {V}  \in \mathbb{R}^{I_{1}\times \ldots \times I_{N}\times   J_{1}\times \ldots \times J_{M}}$ and $\mathcal {U}  \in \mathbb{R}^{J_{1}\times \ldots \times J_{M}\times   I_{1}\times \ldots \times I_{M}}$ be given. Then, the global Golub--Kahan bidiagonalization (GGKB) algorithm is summarized in Algorithm \ref{alg2}.

\begin{algorithm}[!h]
	\caption{Global Golub--Kahan algorithm via Einstein product}\label{alg2}
	\begin{enumerate}
\item {\bf Inputs} The tensors $\mathcal {A}$, $\mathcal {C}$,  and an integer $\ell$.
\item 	Set $\sigma_1= \Vert \mathcal{C} \Vert_F$, 
$\mathcal {U}_1=\mathcal {C}/\sigma_1$ and  $\mathcal {V}_0=0$
\item For $j=1,2,\ldots,\ell$ Do
 \item $\widetilde {\mathcal {V}}= \mathcal {A}^T \ast_N \mathcal {U}_{j} -\sigma_{j}\mathcal {V}_{j-1}$
	\item $\rho_j=\Vert \widetilde {\mathcal {V}}\Vert_F$ if $\rho_j=0$ stop, else
	\item $\mathcal {V}_j=\widetilde {\mathcal {V}}/\rho_j$
	\item $\widetilde {\mathcal {U}}=\mathcal {A} \ast_N \mathcal {V}_j-\rho_j \mathcal {U} _{j}$
	\item $\sigma_{j+1}=\Vert \widetilde {\mathcal {U}} \Vert_F$
	\item if $\rho_j=0$ stop, else
	\item $\mathcal {U}_{j+1}=\widetilde {\mathcal {U}}/\sigma_{j+1}$
\item EndDo
\end{enumerate}

\end{algorithm}

\medskip 

\noindent 
Assume that  $\ell$ steps of the GGKB process have been performed, we form the lower bidiagonal matrix $C_\ell\in\mathbb{R}^{\ell\times \ell}$ 
 \[
 C_\ell=\begin{bmatrix}
 \rho_1\\
 \sigma_2 & \rho_2&\\
 &\ddots&\ddots\\
 & & \sigma_{\ell-1}&\rho_{\ell-1}\\
 &&&\sigma_\ell&\rho_\ell
 \end{bmatrix}
 \]
 and 
 \[
 \widetilde{C}_\ell=\begin{bmatrix}
 C_\ell\\
 \sigma_{\ell+1}e_\ell^T
 \end{bmatrix}\in\mathbb{R}^{(\ell+1)\times \ell}.
 \]

\begin{proposition}\label{prop4.3}
Assume that $\ell$ have performed and all non-trivial entries of the matrix $\widetilde{C}_{\ell}$ are positive.  Let $\mathbb{V}_\tau$ and $\mathbb{U}_\tau$  be $(M+N+1)$-mode tensors 
whose frontal slices are given by $\mathcal{V}_j$ and $\mathcal{U}_j$ for $j=1,2,\ldots,\tau$, respectively. Furthermore, suppose that $\mathbb{W}_\tau$ and  $\mathbb{W}_\tau^*$ are $(M+N+1)$-mode tensors 
having frontal slices $\mathcal {A} \ast_N\mathcal{V}_j$ and $\mathcal {A}^T \ast_N\mathcal{U}_j$ for $j=1,2,\ldots,\tau$, respectively. 
The following relations hold:
 \begin{eqnarray}
  \mathbb{W}_\ell& = &\mathbb{U}_{\ell+1} \times_{(M+N+1)} {\widetilde C}_\ell^T,\label{eq4.13}\\
 \mathbb{W}_\ell^*	& = &\mathbb{V}_{\ell} \times_{(M+N+1)} { C}_\ell^T.\label{eq4.14}
 \end{eqnarray} 
\end{proposition}
\begin{proof}
	From Lines 7 and 10 of Algorithm \ref{alg2}, we have	
\[
\mathcal {A} \ast_N\mathcal{V}_j=\rho_j\mathcal{U}_j+ \sigma_{j+1}\mathcal{U}_{j+1}\qquad j=1,2\ldots,\ell
\]
which conclude \eqref{eq4.13} from definition of $n$-mode product. Similarly, Eq. \eqref{eq4.14} follows from Lines 4 and 6 of Algorithm \ref{alg2}.
	\end{proof}

\medskip
\noindent 
Here, we apply the following Tikhonov regularization approach and solve the new problem

\begin{equation}\label{tikho3}
\displaystyle \min_{\mathcal{X}} \left(\Vert \mathcal{A} \ast_N \mathcal{X}-\mathcal{C} \Vert_F^2 + \mu^{-1} \Vert \mathcal{C} \Vert_F^2\right),
\end{equation}
 We comment on the use of  $\mu^{-1} $
in (\ref{tikho3}) instead of $\mu$  below. As for the iterative tensor Global GMRES method discussed in the previous subsection, 
the computation of an accurate approximation $\mathcal{X}_\mu$ requires that a suitable
value of the regularization parameter  be used.  In this subsection, we use the discrepancy principle to determine a suitable regularization parameter assuming that an approximation of the norm of additive error is available, 	i.e., we have a bound $\varepsilon$ for $\|\mathcal{E}\|_F$. This priori information suggests that $\mu$ has to be determined such that,
\begin{equation}\label{discrepancy}
\|\mathcal{C}-\mathcal{A} \ast_N \mathcal{X}_\mu\|_F=\eta\epsilon,
\end{equation}
where $\eta>1$ is the safety factor for the discrepancy principle. A zero-finding method can be used to solve (\ref{discrepancy}) in order to find a suitable regularization parameter which also implies that $\|\mathcal{C}-\mathcal{A} \ast_N \mathcal{X}_\mu\|_F$ has to be evaluated for several $\mu$-values. When the tensor  $\mathcal{A}$ is of moderate size, the quantity $\|\mathcal{C}-\mathcal{A} \ast_N \mathcal{X}_\mu\|_F$ can be easily evaluated. This computation becomes expensive when $\mathcal{A}$ is a large tensor, which means that its evaluation by a  zero-finding method can be very difficult and computationally expensive. In what follows, it is shown that this  difficulty can be remedied by using a connection between the   Golub--Kahan bidiagonalization (GGKB) and  Gauss-type quadrature rules. This connection provides approximations of moderate sizes to the quantity  $\|\mathcal{C}-\mathcal{A} \ast_N \mathcal{X}_\mu\|_F$ and therefore gives a solution method to inexpensively solve (\ref{discrepancy}) by evaluating these small quantities; see \cite{belguide, belguide2} for discussion on this method.\\
Let us consider the following functions of $\mu$,
\begin{eqnarray}\label{Gkfmu}
\phi(\mu)&=&\left\|\mathcal{C}-\mathcal{A} \ast_N \mathcal{X}_\mu\right\|_{F}^{2}\\
\mathcal{G}_\ell f_\mu&=&\|\mathcal{C}\|_F^2 e_1^T(\mu C_\ell C_\ell^T+I_\ell)^{-2}e_1,\\
{\mathcal R}_{\ell+1}f_\mu&=&\|\mathcal{C}\|_F^2 e_1^T(\mu \widehat{C}_\ell\widehat{C}_\ell^T+I_{\ell+1})^{-2}e_1;
 \end{eqnarray}
 $\mathcal{G}_lf$ and ${\mathcal R}_{\ell+1}f_\mu$ are pairs of Gauss and Gauss-Radau quadrature rules, respectively, and they approximate $\phi(\mu)$ as follows
 \begin{equation}
\mathcal{G}_\ell f_\mu\leq\phi(\mu)\leq{\mathcal R}_{\ell+1}f_\mu
 \end{equation}
 As shown in \cite{belguide, belguide2}, for a given value of $l\geq 2$, we solve for $\mu$ the nonlinear equation
 \begin{equation}\label{lin22}
 {\mathcal G}_\ell f_\mu=\epsilon^2
 \end{equation}
 by using Newton's method.\\
  The  use the parameter $\mu$ in (\ref{tikho3}) instead of $1 / \mu,$ implies that the left-hand side of (\ref{discrepancy}) is a decreasing convex function of $\mu .$ Therefore, there is a unique solution, denoted by $\mu_{\varepsilon},$ of
$$
\phi(\mu)=\varepsilon^{2}
$$
for almost all values of $\varepsilon>0$ of practical interest and therefore also of (\ref{lin22}) for $\ell$ sufficiently large; see \cite{belguide, belguide2} for analyses.  We accept $\mu_\ell$ that solve (\ref{discrepancy}) as an approximation of $\mu$, whenever we have
 \begin{equation}\label{upperbd}
 {\mathcal R}_{\ell+1}f_{\mu}\leq\eta^2\epsilon^2. 
 \end{equation}
 If (\ref{upperbd}) does not hold for $\mu_l$, we carry out one more GGKB steps, replacing $\ell$ by $\ell+1$ and solve the nonlinear equation
 \begin{equation}\label{}
 {\mathcal G}_{\ell+1}f_\mu=\epsilon^2;
 \end{equation}
 see \cite{belguide, belguide2} for more details. Assume now that (\ref{upperbd}) holds for some $\mu_\ell$. The corresponding regularized solution is then computed by
 \begin{equation}\label{Xkmu}
\mathcal {X}_\ell=\mathbb{V}_\ell \bar{\times}_{(M+N+1)} y_\ell,
 \end{equation}
 where $y_{\mu_\ell}$ solves 
 \begin{equation}\label{normeq2}
 (\bar{C}_\ell^T\bar{C}_\ell+\mu_\ell^{-1} I_l)y=\sigma_1\bar{C}_\ell^Te_1,\qquad\sigma_1=\|\mathcal{C}\|_F.
 \end{equation}
 It is also computed by solving the least-squares problem
 \begin{equation}\label{leastsq}
 \min_{y\in\mathbb{R}^\ell} \begin{Vmatrix}
 \begin{bmatrix}
 \mu_\ell^{1/2}\bar{C}_\ell\\
 I_\ell
 \end{bmatrix}
 y-\sigma_1\mu_\ell^{1/2}e_1 \end{Vmatrix}_2.
 \end{equation}The following result shows an important property of the approximate solution (\ref{Xkmu}). We include a proof for completeness.

 \begin{proposition}
 		Under  assumptions of Proposition \ref{prop4.3}, let $\mu_{\ell}$ solve (\ref{lin22}) and let $y_{\mu_{\ell}}$ solve (\ref{leastsq}). Then the associated approximate solution (\ref{Xkmu}) of (\ref{tikho3}) satisfies
$$
\left\|\mathcal{A}\ast_N \mathcal{X}_{\mu_{\ell}}-\mathcal{C}\right\|_{F}^{2}=R_{\ell+1} f_{\mu_{\ell}}
$$
  \end{proposition}
  \begin{proof}
By Eq. \ref{eq4.13}, we have
\begin{eqnarray*}
\mathcal{A}\ast_N \mathcal{X}_{\mu_{l}}=\sum_{i=1}^{\ell}(\mathcal{A}\ast_N \mathcal{V}_{i}) y_\ell^{i}&=&\mathbb{W}_{\ell}\bar{\times}_{(M+N+1)}y_\ell \\
&=&\mathbb{U}_{\ell+1}\bar{\times}_{(M+N+1)}(\widetilde{C}_\ell y_\ell)
\end{eqnarray*}
Using the above expression gives
$$
\begin{aligned}
\left\|\mathcal{A}\ast_N \mathcal{X}_{\mu_{l}, \ell}-\mathcal{C}\right\|_{F}^{2}&=\left\|\mathbb{U}_{\ell+1}\bar{\times}_{(M+N+1)}(\widetilde{C}_\ell y_\ell)-\sigma_1\mathcal{U}_1\right\|_{F}^{2} \\
&=\left\|\mathbb{U}_{\ell+1}\bar{\times}_{(M+N+1)}(\widetilde{C}_\ell y_\ell)-\mathbb{U}_{\ell+1}\bar{\times}_{(M+N+1)}(\sigma_1e_1)\right\|_{F}^{2} \\
&=\left\|\mathbb{U}_{\ell+1}\bar{\times}_{(M+N+1)}\left(\widetilde{C}_\ell y_\ell-\sigma_1e_1\right)\right\|_{F}^{2}\\
&=\left\|\mathbb{U}_{\ell+1}\boxtimes^{(M+N+1)}(\mathbb{U}_{\ell+1}\bar{\times}_{(M+N+1)}\left(\widetilde{C}_\ell y_\ell-\sigma_1e_1\right))\right\|_{F}^{2}\\
&=\left\|\left(\mathbb{U}_{\ell+1}\boxtimes^{(M+N+1)}\mathbb{U}_{\ell+1}\right)\left(\widetilde{C}_\ell y_\ell-\sigma_1e_1\right))\right\|_{2}^{2}\\
&=\left\|\widetilde{C}_\ell y_\ell-\sigma_1e_1\right\|_{2}^{2}
\end{aligned}
$$
where we recall that $\sigma_{1}=\|\mathcal{C}\|_{F}$. We now express $y_{\mu_{\ell}}$ with the aid of  (\ref{normeq2}) and apply the following  identity 
$$I-A\left(A^{T} A+\mu^{-1} I\right)^{-1} A^{T}=\left(\mu A A^{T}+I\right)^{-1}$$
with $A$ replaced by $\widehat{C}_{\ell},$ to obtain
$$
\begin{aligned}
\left\|\mathcal{A}\ast_N \mathcal{X}_{\mu_{l}, \ell}-\mathcal{C}\right\|_{F}^{2} &=\sigma_{1}^{2}\left\|e_{1}-\widetilde{C}_{\ell}\left(\widetilde{C}_{\ell}^{T} \widetilde{C}_{\ell}+\mu_{\ell}^{-1} I_{\ell}\right)^{-1} \widetilde{C}_{\ell}^{T} e_{1}\right\|_{F}^{2} \\
&=\sigma_{1}^{2} e_{1}^{T}\left(\mu_{\ell} \widetilde{C}_{\ell} \widetilde{C}_{\ell}^{T}+I_{\ell+1}\right)^{-2} e_{1} \\
&=R_{\ell+1} f_{\mu_{\ell}}
\end{aligned}
$$
which conclude the assertion.
  \end{proof}
  
\noindent  The following algorithm summarizes the main steps to compute a regularization parameter and a corresponding regularized solution of (\ref{eq1}) using GGKB and quadrature rules method for Tikhonov regularization.
 	  
 \begin{algorithm}[!h]
	\caption{GGKB and quadrature rules method for Tikhonov regularization via Einstein product}\label{TG-GK}
	\begin{enumerate}
\item {\bf Inputs} Tensors $\mathcal {A}$, $\mathcal {C}$, $\eta\leq 1$  and  $\varepsilon$.
\item Determine the orthonormal bases $\mathbb{U}_{l+1}$ and $\mathbb{V}_{l}$ of tensors, and the bidiagonal matrices $C_\ell$ and $\widetilde{C}_\ell$
by implementing Algorithm \ref{alg2}.
\item Determine $\mu_{\ell}$ that satisfies (\ref{lin22}) with Newton's method.
\item  Determine $y_{\mu_{\ell}}$ by solving  (\ref{leastsq}) and then compute $X_{\mu_{\ell}}$ by (\ref{Xkmu}).
\end{enumerate}

\end{algorithm}
\section{Numerical results}\label{sec4}
	    This section provides some numerical results to show the performance of Algorithms \ref{alg11} and Algorithm \ref{TG-GK}
	    when applied to the restoration of blurred and noisy color images and videos. For clarity and definiteness, we first focus on the formulation of a tensor model, describing the blurring that is taking place in the process of going from the exact to the blurred RGB image (or video). Notwithstanding what has just been said, recovering RGB (or video) from their blurry and noisy observations can be seen as a tensor problem of the form (\ref{eq1}). Therefore, it’s very important to understand how the model (\ref{eq1}) can be constructed for RGB images and color video deblurring problems. In what follows, we will concentrate only on the formulation of the tensor model for RGB image deblurring problems and will comment at the end of this section how a similar one can be formulated for color video deblurring problems. We recall that an RGB image is just multidimensional array of dimension $M\times N\times 3$ whose entries are the light intensity. Throughout this paper, we assume that the original RGB image has the same dimensions as the blurred one, and we refer to it as $N\times N\times 3$ tensor. Let  $\mathcal C$ represent the available blurred RGB image,  let $\mathcal X$ denote the desired unknown blurred RGB, and let $\mathcal A$ be the tensor describing  the blurring that is taking place in the process of going from $\mathcal X$ to $\mathcal C$. It is well known in the literature of image processing that all the blurring operators can be characterized by a Point Spread Function (PSF) describing the blurring process and the boundary conditions outside the image, see \cite{HNO}. Once the two-dimensional PSF array, $P$, is specified, we can as well build the blurring tensor $\mathcal A$. By using the fact that the blurring process of an RGB image is simply  a multi-dimensional convolution operation of the PSF array  $P$ and the original three-dimensional image $\mathcal X$,  the blurring tensor $\mathcal A$ can be easily constructed  by placing the elements of $P$ in the appropriate positions. Note that the PSF is a two-dimensional array $P$ describing the image of a single white pixel, which makes its dimensions much smaller than $N$. Therefore, $P$ contains all the required information about the blurring throughout the RGB image $\mathcal C$. To illustrate this, the discrete operation for multi-dimensional convolution using a $3\times 3$ local and spatially invariant PSF array $P$ with $p_{22}$ is its center,  and assuming zero boundary conditions, is given by:
      \begin{eqnarray}\label{conv}
      \mathcal{C}_{ijk}&=&p_{33}\mathcal X_{i-1j-1k}+p_{32}\mathcal X_{i-1jk}+p_{31}\mathcal X_{i-1j+1k}+p_{23}\mathcal X_{ij-1k}+p_{22}\mathcal X_{ijk}\\&+&p_{21}\mathcal X_{ij+1k}+p_{13}\mathcal X_{i+1j-1k}+p_{12}\mathcal X_{i+1jk}+p_{11}\mathcal X_{i+1j+1k},
   \end{eqnarray}
 for $i,j=1,...,N$ and $k=1,2,3.$ Here the zero boundary conditions are imposed so the values of $\mathcal X$ are zero outside the RGB image, i.e., $\mathcal{X}_{i0k}=\mathcal{X}_{iN+1k}=\mathcal{X}_{0jk}=\mathcal{X}_{N+1jk}=0$  for $0<i,j<N+1$ and $k=1,2,3.$
  By using  Definition \ref{defpari} and Definition \ref{defblock} a fourth order tensor $\mathcal A\in \mathbb{R}^{N \times N \times N \times N}$ associated with (\ref{conv}),  with partition $\left(1,N, 1, N\right)$,  can be partitioned into matrix blocks of size $N \times N$. Each block is denoted by $\mathcal A_{i_{2}, i_{4}}^{(2,4)}=$ $\mathcal A\left(:, i_{2}, :, i_{4}\right) \in \mathbb{R}^{N \times N}$ with $i_{2}=1, \ldots, N$ and $i_{4}=1, \ldots, N.$  The nonzero entries of the matrix block $\mathcal A_{a, b}^{(2,4)}\in \mathbb{R}^{N \times N}$ are given by
 
 $$\begin{array}{ll}
    (\mathcal A_{a, b}^{(2,4)})_{a-1b-1}= p_{33}; & (\mathcal A_{a, b}^{(2,4)})_{ab+1}= p_{21} \\
      (\mathcal A_{a, b}^{(2,4)})_{a-1b}= p_{32};&  (\mathcal A_{a, b}^{(2,4)})_{a+1b-1}= p_{13}  \\
      (\mathcal A_{a, b}^{(2,4)})_{a-1b+1}= p_{31};&(\mathcal A_{a, b}^{(2,4)})_{a+1b}= p_{12}  \\
      (\mathcal A_{a, b}^{(2,4)})_{ab-1}= p_{23};&(\mathcal A_{a, b}^{(2,4)})_{a+1b+1}= p_{11}  \\
      (\mathcal A_{a, b}^{(2,4)})_{ab}= p_{22}&
 \end{array}$$for $a, b=2, \dots, N-1.$\\
The first following examples applies  Algorithms \ref{alg11} and \ref{TG-GK} to
	    the restoration of blurred color image and video that have been contaminated by 
	    Gaussian blur and by additive zero-mean white
	    Gaussian noise. We consider the blurring to be local and spatially invariant. In this the case the entries of the Gaussian PSF array $P$ are given by
	    $$p_{i j}=\exp \left(-\frac{1}{2}\left(\frac{(i-k)}{\sigma}\right)^{2}-\frac{1}{2}\left(\frac{(j-\ell)}{\sigma}\right)^{2}\right),$$
	    where $\sigma$ controls the width of the Gaussian PSF and $(k,\ell)$ is its center, see \cite{HNO}. Note  that $\sigma$ controls the amount of smoothing, i.e. the larger the $\sigma$,  the more ill posed the problem. The original tensor image is denoted by $\widehat{\mathcal{X}}$ in each example and $\mathcal{A}$ represents the blurring tensor. The tensor $\widehat{\mathcal{C}}=\mathcal{A}\ast_N\widehat{\mathcal{X}}$ represents the associated blurred and noise-free multichannel image. We generated a blurred and noisy tensor image $\mathcal{C}=\widehat{\mathcal{C}}+\mathcal{N},$ where $\mathcal{N}$ is a noise tensor with normally distributed random entries with zero mean and with variance chosen to correspond to a specific noise level $\nu:=\|\mathcal{N}\|_F /\|\widehat{\mathcal{C}}\|_F.$
	    To determine the effectiveness of our solution methods, we evaluate 
	    $$\text{RE}=\frac{\left\|\hat{X}-X_{\textbf{restored}}\right\|_{F}}{\|\hat{X}\|_{F}}$$
	    and the Signal-to-Noise
	    Ratio (SNR) defined by
	    \[\text{SNR}(X_{\text{restored}})=10\text{log}_{10}\frac{\|\widehat{X}-E(\widehat{X})\|_F^2}{\|X_{\textbf{restored}}-\widehat{X}\|_F^2}\]
	    where $E(\widehat{X})$ denotes the mean gray-level of the uncontaminated image $\widehat{\mathcal{X}}$. 
	    All computations were carried out using the MATLAB environment on an Intel(R) Core(TM) i7-8550U CPU @ 1.80GHz (8 CPUs) computer with 12 GB of
	    RAM. The computations were done with approximately 15 decimal digits of relative
	    accuracy. 
	    \subsection{Example 1}
	    This example illustrates the performance of Algorithms \ref{alg11} and \ref{TG-GK} 4 when applied to the 
	    restoration of 3-channel RGB color image that have been contaminated by Gaussian blur and additive noise. The original (unknown) $\mathrm{RGB}$ image $\widehat{\mathcal X} \in \mathbb{R}^{256 \times 256 \times 3}$ is the \textbf{papav256} image from \textbf{MATLAB}. It is shown on the left-hand side of Figure \ref{fig1}. For the blurring tensor $\mathcal{A}$,  we consider a PSF array $P$ with $\sigma=2$ under zero boundary conditions. The associated blurred and noisy RGB image $\widehat{\mathcal{C}}=\mathcal{A}\ast_N\widehat{\mathcal{X}}$ is shown on the right-hand side of  Figure \ref{fig1}. The noise level is $\nu=10^{-3}$. Given the contaminated RGB image $\mathcal{C}$, we would like to recover an approximation of the original RGB image $\widehat{\mathcal X}$. Table \ref{tab1} compares, the computing time (in seconds),  the relative errors and the PSNR of the computed restorations. Note that in this table, the allowed maximum number of outer iterations for Algorithm \ref{alg11}  with noise level $\nu=10^{-2}$ was 4.  The restoration for  noise level $v=10^{-3}$ is shown on the left-hand side of Figure \ref{fig2} and it is obtained by applying  Einstein tensor global GMRES method (Algorithm
\ref{alg11}) with input $\mathcal{A}$, $\mathcal{C}$, $\mathcal{X}_0=\mathcal{O}$, $\varepsilon=10^{-6}$, $m=10$ and $\textbf{Maxit}=10$. Using GCV, the computed optimal value for the projected problem in Algorithm 2 was $\mu_5=9.44\times 10^{-4}.$  The restoration obtained with Algorithm \ref{TG-GK} is shown on the right-hand side of Figure \ref{fig2}. The discrepancy principle with $\eta=1.1$ is satisfied when $\ell=61$ steps of the Einstein tensor GGKB method have been carried out, producing a regularization parameter given by $\mu_\ell=2.95\times 10^{-4}$.
\begin{table}[htbp]
	\caption{Results for Example 1.}\label{tab1}
	\begin{center}
		\begin{tabular}{lcccc}
\hline Noise level & Method & PSNR & RE & CPU-time (seconds) \\
\hline {$10^{-3}$}& Algorithm 2 & 21.76 & $6.09 \times 10^{-2}$ & 8.28 \\  & Algorithm 4 & 24.37 & $4.51 \times 10^{-2}$ & 7.29 \\
\hline  {$10^{-2}$} & Algorithm 2 & 20.60 & $6.96 \times 10^{-2}$ & 3.31 \\
& Algorithm 4 & 20.97 & $6.67 \times 10^{-2}$ & 1.58 \\
\hline
\end{tabular}
	\end{center}
\end{table}
\begin{figure}
\begin{center}
\includegraphics[width=5in]{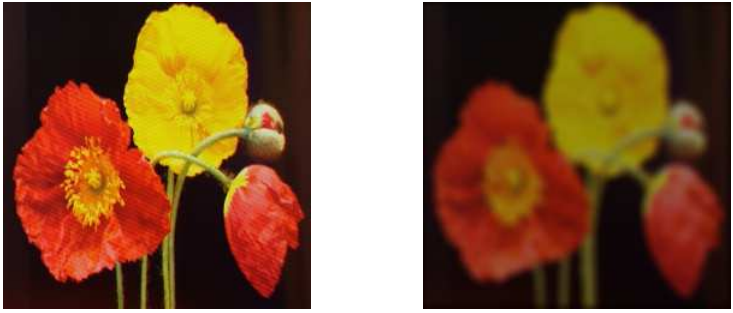}
\caption{Example 1: Original image (left), blurred and noisy image (right).}\label{fig1}
\end{center}
\end{figure}

\begin{figure}
\begin{center}
\includegraphics[width=5in]{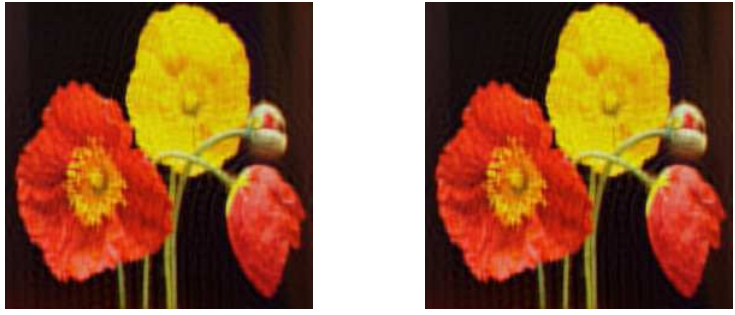}
\caption{Example 1: Restored image by Algorithm 4 (left), and restored image by Algorithm 2 
(right).}\label{fig2}
\end{center}
\end{figure}
\subsection{Example 2}
In this example, we evaluate the effectiveness of Algorithms \ref{alg11} and \ref{TG-GK} when applied to the restoration of a color video defined by a sequence of RGB images. Video restoration is the problem of restoring a sequence of $k$ color images (frames). Each frame is represented by a tensor of $N \times N\times3$ pixels. In the present example, we are interested in restoring 10 consecutive frames of a contaminated video. We consider the xylophone video from MATLAB. The video clip is in MP4 format with each frame having $240 \times 240$ pixels. The (unknown) blur- and noise-free frames are stored in the tensor $\widehat{\mathcal{C}} \in \mathbb{R}^{N \times N\times3\times10}$. These frames are blurred by a blurring tensor $\mathcal A$ of the same kind and with the same parameters as in the previous example. Figure \ref{frame5ob} shows the 5th exact (original) frame and the contaminated version, which is to be restored. Blurred and noisy frames are generated by $\widehat{\mathcal{C}}=\mathcal{A}\ast_N\widehat{\mathcal{X}}$ where the tensor $\mathcal{E}$ represents white Gaussian noise of levels $\nu=10^{-3}$ or $\nu=10^{-2}$. Table \ref{tab2} displays the performance of algorithms. For Algorithm \ref{alg11}, we have used as an input  $\mathcal{A}$, $\mathcal{C}$, $\mathcal{X}_0=\mathcal{O}$, $\varepsilon=10^{-6}$, $m=10$ and $\textbf{Maxit}=10$. For the ten outer iterations,  minimizing the GCV function  produces  $\mu_{10}=9.44 \times 10^{-4}$. Using Algorithm \ref{TG-GK},  the discrepancy principle with $\eta=1.1$ have been satisfied after $\ell=59$ steps of the Einstein tensor GGKB method, producing a regularization parameter given by $\mu_\ell=1.06\times10^{-4}$. The restorations obtained with Algorithms \ref{alg11}  and \ref{TG-GK} are shown on the left-hand and right-hand sides of Figure \ref{frame5r}, respectively.
\begin{table}[htbp]
	\caption{Results for Example 2.}\label{tab2}
	\begin{center}
		\begin{tabular}{lcccc}
\hline Noise level & Method & PSNR & Relative error & CPU-time (second) \\
\hline {$10^{-3}$}& Algorithm 2 & 15.48 & $6.84 \times 10^{-2}$ & 38.93 \\  & Algorithm 4 & 19.24 & $4.43 \times 10^{-2}$ & 27.37 \\
\hline  {$10^{-2}$} & Algorithm 2 & 14.50 & $7.65 \times 10^{-2}$ & 15.55 \\
& Algorithm 4 & 15.13 & $7.11 \times 10^{-2}$ & 4.40 \\
\hline
\end{tabular}
	\end{center}
\end{table}
	\begin{figure}
		\begin{center}
			\includegraphics[width=5in]{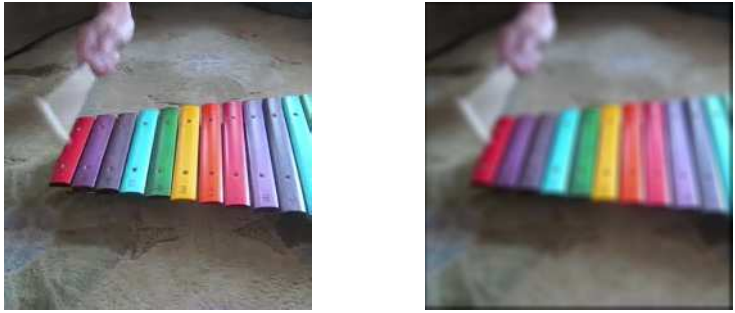}
			\caption{Frame no. 5: Original frame  (left), blurred and noisy frame  (right).
				}\label{frame5ob}
		\end{center}
	\end{figure}
\begin{figure}
	\begin{center}
		\includegraphics[width=5in]{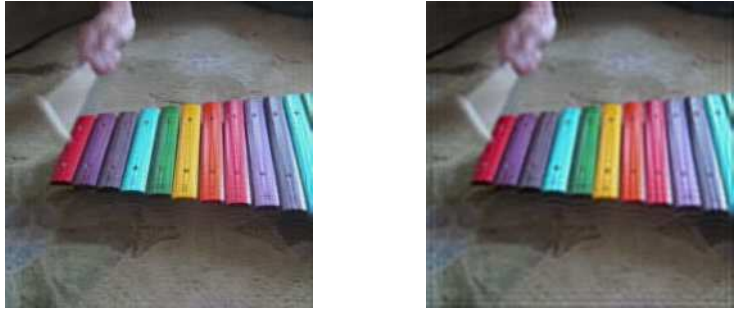}
		\caption{Frame no. 5: Restored frame by Algorithm 4 (left), and restored frame by Algorithm 2 
			(right).}\label{frame5r}
	\end{center}
\end{figure}
 	
 \section{Conclusion}\label{sec5}
We extended the GMRES and Gloub--Kahan bidiagonalization  in conjunction
of Tikhonov regularization for solving (possibly) ill-conditioned multilinear systems via Einstein product with perturbed right-hand side. Numerical experiments were disclosed for image and video processing to demonstrate
the feasibility of proposed iterative algorithms.

 \end{document}